\documentclass{amsart}

 \usepackage{amsmath, amscd, graphicx, amsbsy, pstricks, amssymb}
\usepackage{epsfig}
\usepackage[cmtip,all]{xy}

\usepackage{hyperref}

\newcommand{\MathVector}[1]{\boldsymbol{#1}}
\newcommand{\intersect}[1]{\bigl< #1 \bigr>}

\newcommand{\FiniteCount}[1]{\#\left| #1 \right|}
\newcommand{\norm}[1]{|| #1 ||}

\newcommand{\LieAlg}[1]{\mathfrak{#1}}
\newcommand{\Module}[1]{\mathfrak{#1}}

\newcommand{\Moduli}{\mathcal{M}}
\newcommand{\CompactModuli}{\overline{\mathcal{M}}}
\newcommand{\Teichmuller}{\mathcal{T}}
\newcommand{\CompactTeich}{\overline{\mathcal{T}}}

\newcommand{\Integers}{\mathbb{Z}}
\newcommand{\Reals}{\mathbb{R}}
\newcommand{\Complex}{\mathbb{C}}
\newcommand{\AltOmega}{\hat{\Omega}}
\DeclareMathOperator{\MCG}{Mod}
\DeclareMathOperator{\Stab}{Stab}

\DeclareMathOperator{\Aut}{Aut}
\newtheorem{theorem}{Theorem}[section]
\newtheorem{proposition}[theorem]{Proposition}
\newtheorem{corollary}[theorem]{Corollary}
\newtheorem{lemma}[theorem]{Lemma}

\numberwithin{equation}{section}

\allowdisplaybreaks

\begin{document}

\author{Brad Safnuk}
\address{McMaster University, Hamilton, Canada}
\email{bsafnuk@math.mcmaster.ca}
\title[Integration on $\CompactModuli_{g,n}$ through localization]{Integration on moduli spaces of stable curves through localization}
\begin{abstract}
We introduce a new method of calculating intersections on $\CompactModuli_{g,n}$, using localization of equivariant cohomology. As an application, we give a proof of Mirzakhani's recursion relation for calculating intersections of mixed $\psi$ and $\kappa_1$ classes.
\end{abstract}

\date{\today}
\maketitle

\section{Introduction}

Many integration formulas that exist for the moduli space of stable curves seem to obey a type of localization principle: that is, an integral over $\CompactModuli_{g,n}$ can often be written in terms of integrals over lower dimensional moduli spaces, which naturally embed in $\CompactModuli_{g,n}$. The primary example is the Virasoro relations coming from the Witten-Kontsevich Theorem, where integrals of tautological classes can be written as integrals over moduli spaces obtained by pinching off 3--punctured spheres from the base surface. This geometrically appealing fact was even noticed by Witten in his groundbreaking work on the subject \cite{art:Witten2dGravity}. However, it has not been possible to prove any such statement using Atiyah--Bott localization directly. In fact, the moduli space of stable curves has very little known symmetry, so there are no contenders for a group action which would lead to localization. 

We present a construction which allows the calculation of arbitrary integrals on $\CompactModuli_{g,n} $using localization techniques. The group action is not on $\CompactModuli_{g,n}$, but rather on certain infinite covers $\CompactModuli^{\Gamma}_{g,n}$ first considered by Mirzakhani~\cite{art:MirzakhaniWP} for calculating the Weil-Petersson volume of $\Moduli_{g,n}$. 

The technique presented in this paper is very much in the spirit of (and heavily influenced by) Mirzakhani's approach. However, it has the advantage of not being reliant on properties of the Weil--Petersson form - in principal it works for any integral on $\CompactModuli_{g,n}$.  Moreover, the results do not directly involve the moduli spaces of bordered Riemann surfaces, which played an essential role in Mirzakhani's work. This makes it plausible to apply this localization theory in a broader context. For example, it would be interesting to extend these ideas to the moduli space of stable maps appearing in Gromov--Witten theory or the moduli space of $r$-spin structures. Each of these spaces has conjectural intersection formulas (the Virasoro conjecture in the case of the former and the generalized Witten conjecture for the latter) which are remarkably similar in form to the original Witten-Kontsevich theorem. This suggests that a proof should come from the geometry of $\CompactModuli_{g,n}$.

In Section~\ref{sect:Local} we construct the necessary generalization of the Atiyah--Bott localization theorem to apply it in the present context. The first (minor) complication is that the spaces involved are not compact. However, they behave well enough that extending localization is not a problem. More serious, is that the equivariant forms which naturally arise when calculating intersections of tautological classes are not in the usual ring of the Cartan model  $\Omega(M) \otimes S(\LieAlg{g}^*)$. Instead, they live in $\AltOmega(M)\otimes \Module{m}$, where $\Module{m}$ is a $S(\LieAlg{g}^*)$--module of formal power series in $\LieAlg{g}^*$ and their inverses,  and $\AltOmega(M)$ is the space of  measurable forms smooth on some open dense subset of $M$.  This module over the Cartan ring retains enough features of smooth equivariant forms to admit a type of localization theorem.

Section~\ref{sect:Integrate} outlines the general localization method for $\CompactModuli_{g,n}$. Following Mirzakhani, we use McShane's identity to lift integrals over $\CompactModuli_{g,n}$ to infinite covers with torus symmetry. The fixed point sets are themselves moduli spaces that reside naturally in the boundary of $\CompactModuli_{g,n}$. 

Finally, in Section~\ref{sect:Equiv} we apply the localization construction to calculate intersection numbers of mixed $\psi$ and $\kappa_1$ classes. The result is a recursive formula which is equivalent to the original recursion relation discovered by Mirzakhani \cite{art:MirzakhaniWP}. 

\section{Equivariant cohomology with module coefficients}
\label{sect:Local}

Let $G = T^r$ be an $r$--dimensional torus with Lie algebra $\LieAlg{g}$, and $\Module{m}$ be a graded $S(\LieAlg{g}^*)$ module, where $S(\LieAlg{g}^*)$ denotes the algebra of polynomial functions $\LieAlg{g} \rightarrow \Complex$.
Suppose $M$ is a $2d$--dimensional $G$--manifold. Define the complex $(\Omega_G(M;
\Module{m}), d_G)$ by
\begin{equation*}
  \Omega_G(M; \Module{m}) = \Omega(M)^{G} \otimes \Module{m}
\end{equation*}
with differential
\begin{equation*}
  d_G \omega \otimes s = d\omega \otimes s - \sum_{i=1}^{r}
  \iota_{X_i}\omega\otimes \xi_i \cdot s,
\end{equation*}
where $\{X_i\}$ is any basis for $\LieAlg{g}$ and $\{\xi_i\}$ is the
dual basis of $\LieAlg{g}^*$. By a slight abuse of notation, we denote the induced vector fields on $M$ by $X_i$.

We will also make use of the more general complex $\AltOmega_G(M; \Module{m})$, whereby
$\Omega(M)$ is replaced with $\AltOmega(M)$, the space of differential
forms represented locally by measurable functions, smooth
almost everywhere.  In other words, an element of $\AltOmega(M)$ is a smooth form in $\Omega(U)$ where $U \subset M$ is an open dense subset. Note that we do not assume that differential forms in $\AltOmega(M)$ are integrable.

The above choice of basis for $\LieAlg{g}^*$ gives an isomorphism
$S(\LieAlg{g}^*) \cong \Complex[\xi_1, \ldots, \xi_r]$.  More generally, let $R$ be the ring $\Complex[\xi_1, \ldots, \xi_r]_{(\xi_1, \ldots, \xi_r)}$, whose elements consist of finite sums of monomials in $\xi_i$ and $\xi_i^{-1}$.

For our purposes, it suffices to consider the $R$--module $\Module{m}$
consisting of elements of $\Complex[[\xi_1, \xi^{-1}_1, \ldots, \xi_r,
\xi^{-1}_r]]$ with bounded degree (elements $\xi_i$ have
degree $2$, $\xi^{-1}_i$ degree $-2$). This module is equipped with a
$\Integers^r$ grading which is respected by multiplication by elements
of $R$. In the sequel, any reference to an equivariant form means an element of $\AltOmega_G(M; \Module{m})$, unless otherwise specified.  For any equivariant form $\omega$, we let $\omega_{[k]}$ be the component consisting of forms of (differential) degree $k$. We also define the obvious projection operators $p_{\alpha}:
\Module{m} \rightarrow \Complex$ for any $\alpha \in \Integers^r$, which extend to projections $p_{\alpha}: \AltOmega_G(M;\Module{m}) \rightarrow \AltOmega(M)$.
We note for use in the sequel the following useful statement.
\begin{proposition}
If $\phi:X\rightarrow M$ is a smooth equivariant map and $X$ is compact then $p_{\alpha}\phi^* = \phi^*p_{\alpha}$ and $p_{\alpha}\phi_* = \phi_* p_{\alpha}$ for all smooth, compactly supported equivariant forms.
\end{proposition}

Our aim is to establish a localization theorem in the following
situation. Let $M$ be a compact $G$ manifold, equipped with a
$G$-invariant metric $(\ ,\ )$ for which the vector fields $X_1,
\ldots, X_r$ of the torus action are orthogonal. Furthermore, we assume that if $X$
is a connected component of the fixed point set $M^G$ and $N_X$ is the
corresponding equivariant normal bundle for the embedding of $X$ in
$M$ then there is an equivariant Euler form $e(N_X)$ which splits with respect to
$\{\xi_i \}$. In other words $e(N_X) = e_1 \cdots e_r$ where $e_i$ is
an element of $\Omega(M)^G \otimes \Complex[\xi_i]$.

\begin{theorem}
\label{thm:Localization}
Let $\omega \in \AltOmega_G(M; \Module{m})$ be $d_G$ closed
and homogeneous of degree $2d = \dim M$. If for a fixed
$\alpha\in\Integers^r$ $p_{\alpha + \beta}\omega$ is smooth for all
$\beta \in \Integers^r_{\geq 0}$ (i.e. an element of $\Omega(M)$) then
\begin{equation*}
  \int_{M} p_{\alpha} \omega = \sum_{X} \int_X p_{\alpha} \frac{i^*\omega}{e(N_X)},
\end{equation*}
where the sum is over all connected components of $M^G$ and $e(N_X)$
is a split equivariant Euler form of the normal bundle for the
inclusion $i:X \rightarrow M$.
\end{theorem} 

We note that the splitting hypothesis on the Euler forms is necessary
so that $\frac{i^*\omega}{e(N_X)}$ is an element of
$\AltOmega_G(X; \Module{m})= \AltOmega(X) \otimes \Module{m}$.

The proof of Theorem~\ref{thm:Localization} reduces to several lemmas.

\begin{lemma}
Under the same hypothesis as above, with additionally $\omega$ supported on an invariant contractible neighbourhood $U$ of $X \subset M^G$, and $p_{\alpha + \beta}\omega$ smooth and compactly supported for all $\beta \geq 0$, then
\begin{equation*}
  \int_U p_{\alpha}\omega = \int_X p_{\alpha} \frac{i^*\omega}{e(N_X)}
\end{equation*}
\end{lemma}
\begin{proof}
Let
\begin{equation*}
  e(N_X) = f_r + f_{r-1}\varepsilon_1 + \cdots + f_0\varepsilon_r,
\end{equation*}
where $f_i \in S^i(\LieAlg{g}^*)$ and $\varepsilon_i \in \Omega^{2i}(X)$. Since $G$ is a torus, $f_r \neq 0$, hence $e(N_X)$ is invertible. By the splitting hypothesis
\begin{equation*}
  \frac{1}{e(N_X)} = \frac{1}{f_r} \left(1 + \frac{\varepsilon}{f_r} + \cdots + \frac{\varepsilon^q}{f_r^q} \right)
\end{equation*}
is an element of $\Omega_G(X; R)$, where $\varepsilon = -f_{r-1}\varepsilon_1 - \cdots - f_o\varepsilon_r$.

Let $\pi:U \rightarrow X$ be the projection coming from the identification of $U$ with $N_X$, and $i:X \rightarrow U$ corresponding to the $0$-section. In particular, $\pi^*e(N_X)$ is a Thom form for $X \subset U$, hence $\pi_*\pi^*e(N_X) = 1$.

We arrive at the identity
\begin{equation*}
  \int p_{\alpha}\omega = \int \sum_{\beta+ \gamma = \alpha}  p_{\beta}\Bigl(\frac{\omega}{\pi^* e(N_X)}\Bigr)  p_{\gamma}\bigl(\pi^*e(N_X)\bigr).
\end{equation*}
Because of the splitting hypothesis, any $\gamma = (\gamma_1, \ldots, \gamma_r)$ appearing in the right hand summation has $0\leq \gamma_i \leq 1$.
It follows from the  smoothness assumptions on $\omega$ that $p_{\beta}\Bigl(\frac{\omega}{\pi^* e(N_X)}\Bigr)$ is smooth and compactly supported for all $\beta$ appearing in the above sum. Using the fact that $U$ is contractible to $X$, we then have $p_{\beta}\Bigl(\frac{\omega}{\pi^* e(N_X)}\Bigr)$ cohomologous to $\pi^*i^*p_{\beta}\Bigl(\frac{\omega}{\pi^* e(N_X)}\Bigr)$.
The identity follows by the push-forward formula
\begin{equation*}
	\int_M \pi^*\mu \wedge \nu = \int_X \mu \wedge \pi_*\nu
\end{equation*}
and noting that
\begin{equation*}
  \pi_* p_{\gamma}\tau = p_{\gamma} \pi_* \tau = p_{\gamma} 1 = \delta_{\gamma,0}.
\end{equation*}
\end{proof}

\begin{lemma}
$p_{\alpha} \omega_{[2d]}$ is an exact form on $M \setminus M^G$.
\end{lemma}
\begin{proof}
Define the equivariant 1-form $\theta$ by the equation
\begin{equation*}
  \theta(Z) = \sum_{i=1}^{r} \xi_i^{-1} (X_i, Z).
\end{equation*}
By construction, $d_G \theta$ is invertible outside of $M^G$. If we let
\begin{equation*}
	\nu = -\frac{\theta}{d_G\theta} = \frac{\theta}{D}\left(1 + \frac{d\theta}{D} + \frac{(d\theta)^2}{D^2} + \cdots \right),
\end{equation*}
where $D = \sum_{i=1}^r\norm{X_i}^2$, then $d_G\nu = 1$. Moreover, $\nu$ is an element of $\Omega_G(M\setminus M^G;R)$.
Hence $p_{\alpha}d_G (\nu \wedge \omega) = p_{\alpha}\omega$, and by examining the degrees of the differential forms we have
\begin{equation*}
  d \bigl(p_{\alpha}(\nu\wedge\omega)_{[2d-1]}\bigr) = p_{\alpha}\omega_{[2d]}.
\end{equation*}
\end{proof}
Finally, we prove Theorem~\ref{thm:Localization}.
\begin{proof}
Let $\rho_X$ be a smooth, $G$-invariant function, identically $1$ on a neighbourhood of $X$, $0$ outside of a neighbourhood of $X$.
Define
\begin{equation*}
  \tilde{\omega} = \nu \wedge \omega - \sum_X \rho_X \nu\wedge\omega,
\end{equation*}
which leads to the equation
\begin{equation*}
  \omega = d_G\tilde{\omega} + \sum_X \omega_X,
\end{equation*}
where $\omega_X$ is supported on a neighbourhood of $X$, and equal to $\omega$ sufficiently close to $X$. Furthermore, one can check directly that $p_{\alpha + \beta}\omega_X$ is a smooth compactly supported form for all non-negative $\beta\in\Integers^r$.

One has
\begin{equation*}
  p_{\alpha}\omega = p_{\alpha}d_G\tilde{\omega} + \sum_X p_{\alpha}\omega_X,
\end{equation*}
while
\begin{align*}
  \int_M p_{\alpha} d_G \tilde{\omega} &= \int_M d p_{\alpha}\tilde{\omega}_{[n-1]} \\
  &= 0.
\end{align*}
\end{proof}

It is necessary to extend the above localization theorem to the following noncompact situation. 
\begin{theorem} 
Suppose $M$  is a (possibly noncompact) $G$--manifold with a proper $G$-invariant map $\mu: M \rightarrow
\Reals_{\geq 0}^r$ so that $\mu(M^G)$ is compact. We further assume that $M$ is tame, with an equivariant
embedding $M \rightarrow \overline{M}$.  If $\omega$ is a $d_G$ closed equivariant form on $M$ with
$p_{\alpha+\beta}\omega$ extending by 0 to a smooth closed form on $\overline{M}$ for all $\beta \geq 0$ then
\begin{equation*}
  \int_M p_{\alpha} \omega = \sum_X \int_X p_{\alpha} \frac{i^*\omega}{e(N_X)}.
\end{equation*}
\label{thm:OpenLocalization}
\end{theorem}
\begin{proof}
The hypotheses allow us to construct a closed $G$--manifold by taking the double $2\overline{M} = \overline{M} \cup -\overline{M}$. The equivariant form $\omega$ extends by zero to a form on $2\overline{M}$ meeting the requirements of Theorem~\ref{thm:Localization}. Moreover, the integral of the extension over $2\overline{M}$ agrees with the integral of $\omega$ over $M$.
\end{proof}

\section{Integration on the moduli of curves using localization}
\label{sect:Integrate}

In this section we develop a method of integration for $\CompactModuli_{g,n}$ that uses McShane's identity and localization techniques. We make extensive use of Fenchel-Nielsen coordinates as well as the Teichm\"uller space of marked nodal curves $\CompactTeich_{g,n}$, also called augmented Teichm\"uller space in the literature, so we begin with a description of this space.

A marked nodal curve consists of: a surface $X$ with a complete, finite area hyperbolic metric, a set of simple closed curves $\Gamma = \{\gamma_1, \ldots, \gamma_r \} \subset S_{g,n}$, where $S_{g,n}$ is a fixed surface of genus $g$  and $n$ punctures, and a marking, which is a homeomorphism $f:S_{g,n} \backslash \Gamma \rightarrow X$.  The set of curves $\Gamma$ must be pairwise non-isotopic, with no curve homotopically trivial or boundary parallel (i.e. not collapsible to any puncture). Graphically, one can think of a nodal curve as taking a set of geodesics on a hyperbolic surface of type $(g,n)$ and degenerating the metric so that the geodesics pinch to nodes of length $0$. In their wake, cusps form on either side of the shrinking geodesics. Two marked nodal curves are equivalent if there is an isometry between them which preserves the marking. The set of equivalence classes is denoted $\CompactTeich_{g,n}$. This is a Hausdorff space which contains $\Teichmuller_{g,n}$ as an open dense subset. Refer to \cite{MR590044} for details. Moreover, the mapping class group acts on $\CompactTeich_{g,n}$, with quotient the Deligne--Mumford space $\CompactModuli_{g,n}$. 

Local charts of $\CompactTeich_{g,n}$ are constructed using Fenchel-Nielsen coordinates.   Denote $\mathcal{C} = \Reals \times (0,\infty)$ and $\overline{\mathcal{C}} = \Reals \times [0,\infty) / \sim$ where $(x,0) \sim (y,0)$. Note that the topology on $\overline{\mathcal{C}}$ is coarser than the quotient topology: a basis for open neighbourhoods about the class $(x,0)$ is provided by the sets $\Reals \times [0,\epsilon)$. If $[X,f,\Gamma]$ is an equivalence class of a marked nodal surface, complete $\Gamma$ to a pair of pants decomposition for $S_{g,n}$ by adjoining curves $A = \{\alpha_1, \ldots, \alpha_{3g-3+n-r}\}$.
An open neighbourhood of $(X,f,\Gamma)$ is the set $\overline{\mathcal{C}}^{3g-3+n}$, where we think of an element as a map $h: A \cup \Gamma \rightarrow \overline{\mathcal{}C}$, assigning a twist coordinate and length to each of the curves in the pair of pants decomposition. Specifying a length of 0 means pinching that curve to a node.

In a similar fashion, we may also construct the Teichm\"uller space of marked nodal surfaces with geodesic boundaries. Fixing the lengths of the boundaries to be $L = (L_1, \ldots, L_n)$ results in the spaces $\CompactTeich_{g,n}(L) \rightarrow \CompactModuli_{g,n}(L)$. Allowing the lengths of the boundaries to vary (including shrinking to length 0, which results in a cusp) provides the spaces $\widehat{\Teichmuller}_{g,n}\rightarrow \widehat{\Moduli}_{g,n}$. 

All of the spaces discussed above are symplectic manifolds (with the exceptions of $\widehat{\Teichmuller}_{g,n}$ and  $\widehat{\Moduli}_{g,n}$, which are fibre bundles over $\Reals_{\geq 0}^n$ with symplectic fibres). The symplectic structure is given by the Weil-Petersson form
\begin{equation*}
	\omega_{WP} = \sum_{\gamma} dl_{\gamma} \wedge d\tau_{\gamma}, 
\end{equation*}
where the sum is over any pair of pants decomposition, $l_{\gamma}$ measures the geodesic length of $\gamma$ and $\tau_{\gamma}$ is the twist coordinate, normalized to unit speed. That this is a symplectic form on $\Teichmuller_{g,n}$ is immediate. However, it was proven by Wolpert~\cite{MR796909} that the above expression is invariant under the mapping class group (hence descends to $\Moduli_{g,n}$) and extends smoothly to the boundary. One immediate consequence is that Fenchel--Nielsen twists (the vector fields $\frac{\partial}{\partial\tau_{\gamma}}$) are Hamiltonian, with moment map given by the lengths of the geodesics. Note, however, that Fenchel--Nielsen twists do not descend to a well defined action on moduli space. In particular, the concept of a simple closed geodesic is not well defined for an element of $\Moduli_{g,n}$; instead one has a mapping class group orbit of geodesics to contend with.

To recover a space with a Hamiltonian torus action, we consider the following intermediate cover, first noticed by Mirzakhani~\cite{art:MirzakhaniWP}. Let $\Gamma = \{\gamma_1, \cdots, \gamma_r \}$ be a collection of simple closed curves and $\MCG_{g,n}\cdot \Gamma$ the mapping class group orbit. We define
\begin{align*}
  \Stab \Gamma &= \{[f] \in \MCG_{g,n} \,|\, \text{$f(\gamma_i)$ is homotopic to $\gamma_i$} \} \\
  \CompactModuli_{g,n}^{\Gamma} &= \CompactTeich_{g,n} / \Stab \Gamma \\
    &= \{(X, \MathVector{\eta})\,|\, X \in \CompactModuli_{g,n}, \text{$\MathVector{\eta}$ a collection of geodesics in $\MCG_{g,n}\cdot\Gamma$} \}.
\end{align*}

The space $\CompactModuli^{\Gamma}_{g,n}$ has an $r$-torus action provided by Fenchel-Nielsen twists about the curves $\gamma_i$. We normalize the twists by geodesic length so that the exponential map with parameter 1 gives a full (Dehn) twist about the curve. Let $\{\frac{\partial}{\partial\theta_{i}}\}_{1\leq i\leq r}$ be the normalized vector fields, with corresponding moment map
\begin{align*}
  \mu: \CompactModuli^{\Gamma}_{g,n} &\rightarrow [0,\infty]^r \\
	(X, \eta_1, \ldots, \eta_r) &\mapsto (\frac{1}{2}l(\eta_1)^2, \ldots, \frac{1}{2}l(\eta_r)^2).
\end{align*}
The topology of $\CompactModuli^{\Gamma}_{g,n}$ in a neighbourhood of $\mu^{-1}(\infty, \ldots, \infty)$ is quite complicated; however, a precise understanding is not necessary. Let $\mathfrak{M}^{\Gamma} = \mu^{-1}(\Reals_{\geq 0}^r )$ be the set of nodal surfaces where no geodesic which intersects $\Gamma$ is allowed to degenerate. $\mathfrak{M}^{\Gamma}$ has the structure of an open orbifold and $\mu\bigr|_{\mathfrak{M}^{\Gamma}}$ is proper.  In fact, $\mu^{-1}(a)$ is a torus bundle over the space $\CompactModuli_{S_{g,n}\setminus \Gamma}(L)$, where $L$ is fixed by the lengths of the original boundaries and the new geodesic boundaries coming from cutting along $\Gamma$.  This proves that $\mathfrak{M}^{\Gamma}$ is tame, with an equivariant embedding $\mathfrak{M}^{\Gamma} \rightarrow \overline{\mathfrak{M}}^{\Gamma}$. The fixed point set of the torus action is $\mu^{-1}(0)$, which is the space $\CompactModuli_{S_{g,n}\setminus \Gamma}$. The reason we can restrict attention to the better--behaved space $\mathfrak{M}^{\Gamma}$ is because in the process of lifting integrals from $\CompactModuli_{g,n}$ to $\CompactModuli^{\Gamma}_{g,n}$, we end up with differential forms that vanish to infinite order as the lengths of $\Gamma$ tend to infinity.

Because $\pi:\CompactModuli^{\Gamma}_{g,n}\rightarrow \CompactModuli_{g,n}$ is an (orbifold) covering, we have the following integration formula valid for any function $f: \Reals_+^{n} \rightarrow \Reals$, and form $\omega \in \Omega^*(\CompactModuli_{g,n})$.
\begin{equation*}
  \int_{\CompactModuli_{g,n}} \sum_{\MathVector{\alpha}\in \MCG\cdot\Gamma}f(l(\alpha))\omega = \int_{\CompactModuli_{g,n}^{\Gamma}}f(l(\Gamma))\pi^*\omega.
\end{equation*}
Using McShane's identity, it is possible to write the constant function as a sum over mapping class group orbits of curves, thus enabling the use of the above formalism.

The statement of Mirzakhani's generalization of McShane's identity is as follows. Let $X$ be a hyperbolic surface with $n$ boundary components, labelled $x_1, \ldots, x_n$, of lengths $L_1, \ldots, L_n$. By abuse of notation, a cusp is a boundary component of length 0. Denote by $\mathcal{I}_j$ the set of simple closed geodesics $\gamma$ that separate a pair of pants from $X$ with $x_1$ and $x_j$ at the cuffs, $\gamma$ the waist. Similarly, we denote $\mathcal{J}$ to be the set of pairs of simple closed geodesics $(\alpha, \beta)$ with $(\alpha, \beta, x_1)$ bounding a pair of pants. 
\begin{theorem}[Mirzakhani \cite{art:MirzakhaniWP}] For $X$ as above,
\begin{equation*}
  1 =  \sum_{j=2}^n \sum_{\gamma \in \mathcal{I}_j} \mathcal{R}(L_1, L_j, l(\gamma)) + \sum_{(\alpha,\beta)\in \mathcal{J}}\frac{\FiniteCount{\Aut(\alpha,\beta)}}{2} \mathcal{D}(L_1, l(\alpha), l(\beta)),
\end{equation*}
where
\begin{multline*}
  \mathcal{R}(x,y,z) = \frac{1}{4x} \int_{0}^{x} \Bigl(h(z + t + y) + h(z - t - y) \\
+ h(z + t - y) + h(z - t + y)  \Bigr)dt 
\end{multline*}
\begin{equation*}
  \mathcal{D}(x,y,z) = \frac{1}{2x} \int_{0}^{x} \bigl( h(t + y + z) + h(-t + y + z)  \bigr)dt, 
\end{equation*}
  \begin{equation*}
    h(x) = \frac{2}{1 + e^{x/2}}
  \end{equation*}
and $\Aut(\alpha,\beta)$ is trivial unless $\alpha$ and $\beta$ coincide (which implies that $X$ is a once--punctured torus).
\end{theorem}
\noindent Note that the definitions of $\mathcal{R}(x,y,z)$ and $\mathcal{D}(x,y,z)$ differ slightly from those introduced  by Mirzakhani.

We further refine our subsets of curves by setting $\mathcal{J}_{g_1, \mathcal{A}} \subset \mathcal{J}$ to be the set of pairs of curves $(\alpha, \beta)$ so that $X \setminus \{\alpha, \beta \}$ splits into 3 components: a pair of pants, a surface with genus $g_1$ and boundaries $\{x_i\}_{i\in\mathcal{A}} \cup \alpha$, and a component with genus $g_2 = g - g_1$ and boundaries $\{x_i\}_{i\in \mathcal{A}^c}\cup\beta$. As well, let $\mathcal{J}_{\text{conn}}\subset \mathcal{J}$ be the set of curves with $X \setminus \{\alpha, \beta \}$ splitting into 2 components.
All of the sets of curves $\mathcal{I}_j, \mathcal{J}_{\text{conn}}, \mathcal{J}_{g_1, \mathcal{A}}$ are mutually disjoint. Moreover, the action of the mapping class group is transitive on each of the sets.

Hence, let $\omega$ be any differential form on $\CompactModuli_{g,n}(L)$. We have the following integration formula:
\begin{multline*}
  \int_{\CompactModuli_{g,n}(L)}\omega = \sum_{j=2}^n \int_{\CompactModuli_{g,n}^{\gamma}(L)} \mathcal{R}(L_1, L_j, l(\gamma))\pi^*\omega \\
  + \frac{1}{2}\sum_{\substack{g_1 + g_2 = g \\ \mathcal{A}\coprod\mathcal{B} = \{2, \ldots, n \}}} \int_{\CompactModuli_{g,n}^{\alpha,\beta}(L)}\mathcal{D}(L_1, l(\alpha), l(\beta))\pi^*\omega \\
  + \frac{1}{2}\int_{\CompactModuli_{g,n}^{\mu,\nu}(L)} \mathcal{D}(L_1, l(\mu), l(\nu))\pi^*\omega,
\end{multline*}
where $\gamma$ is a curve in $\mathcal{I}_j$, $(\alpha, \beta)\in \mathcal{J}_{g_1, \mathcal{A}}$, $(\mu,\nu)\in \mathcal{J}_{\text{conn}}$, and $\pi:\CompactModuli^{\Gamma}_{g,n}\rightarrow \CompactModuli_{g,n}$ is the projection map.

In order to make use of the Atiyah--Bott Localization Theorem, we need equivariant forms on $\CompactModuli_{g,n}^{\Gamma}$ whose integrals correspond with the forms appearing in the above sum. Localization would reduce the  integrals to the fixed point set of the torus action, which is $\CompactModuli_{S_{g,n}\setminus \Gamma}$. In the abstract, this is possible by equivariant formality: since the torus action is Hamiltonian there is an isomorphism
\begin{equation*}
  H^*_{G}(\CompactModuli_{g,n}^{\Gamma}) \cong H^*(\CompactModuli_{S_{g,n}\setminus \Gamma}) \otimes S(\LieAlg{g}^*).
\end{equation*}
The above is sufficient to prove that any integral on $\CompactModuli_{g,n}$ can be pushed to a specific subset of the boundary, which can be described using a completely combinatorial language. For more concrete integration formulas, such as the Witten-Kontsevich Theorem, one needs actual equivariant representatives to work with. We take up this topic in the next section. 

To close this section, we compute the equivariant Euler class of the fixed point set  $\CompactModuli_{S_{g,n}\setminus \Gamma}\subset \CompactModuli^{\Gamma}_{g,n}$. Recall that $\Gamma = \{\gamma_1, \ldots, \gamma_r\}$ ($r=1$ or $2$) and we are in the situation where $S_{g,n}\setminus \Gamma$ splits off a pair of pants which contains the first marked point. In particular, the newly formed 3--punctured spheres have no moduli. The remaining components of $S_{g,n}\setminus \Gamma$ have $r$ new punctures added. We denote the $\psi$ classes of these punctures by $\psi_1, \ldots, \psi_r$.
\begin{proposition}
The equivariant Euler class of the fixed point set is 
\begin{equation*}
  e = \prod(-\psi_i + \xi_i).
\end{equation*}
\end{proposition}
\begin{proof}
This statement is an immediate consequence of the fact that the torus acts trivially on the base space and the action on the normal bundle (which is naturally identified with the direct sum of line bundles corresponding to new punctures) is by rotation, opposite to the natural orientation.
\end{proof}

\section{Constructing equivariant forms}
\label{sect:Equiv}
In order to construct equivariant versions of the tautological classes, it is necessary to make use of the symplectic geometry of $\CompactModuli_{g,n}(L)$. 
Suppose $f:\Reals_+ \rightarrow \Reals_+$ is any smooth function so that $f(l(\gamma)){\omega_{WP}^d}$ is an integrable form on $\CompactModuli_{g,n}^{\gamma}(L)$. If we define functions $f_{k}(x)$ by the recursive formula
\begin{align*}
  f_0(x) &= f(x) \\
  \frac{d}{dx} f_k(x) &= xf_{k-1}(x), 
\end{align*}
then the equivariant form
\begin{equation}
  \sum_{k=0}^{d} s^k f_k(l(\gamma)) \frac{\omega_{WP}^{d-k}}{(d-k)!}
  \label{eqn:EquivSingle}
\end{equation}
is an equivariantly closed form in $\Omega_G\bigl(\CompactModuli^{\gamma}_{g,n}(L)\bigr)$. Here one uses the fact that 
\begin{equation*}
\iota_{\frac{\partial}{\partial\theta_{\gamma}}}\omega_{WP} = l_{\gamma} dl_{\gamma},
\end{equation*}
where $\frac{\partial}{\partial\theta_{\gamma}}$ is the vector field generated by a (length normalized) Fenchel--Nielsen twist about $\gamma$.

Similarly, if we define functions $g_{i,j}(x,y)$ by the rules
\begin{align*}
  g_{0,0}(x,y) &= g(x,y) \\
  \frac{\partial}{\partial x}g_{i,j}(x,y) &= xg_{i-1,j}(x,y) \\
  \frac{\partial}{\partial y}g_{i,j}(x,y) &= yg_{i,j-1}(x,y),
\end{align*}
then
\begin{equation}
  \sum_{k=0}^d \sum_{p \in \Integers}s^pt^{-p+k} g_{p, -p+k}(l(\alpha),l(\beta))\frac{\omega_{WP}^{d-k}}{(d-k)!}
  \label{eqn:EquivDouble}
\end{equation}
is an equivariantly closed form in $\AltOmega_G(\CompactModuli^{\alpha,\beta}_{g,n}(L); \Module{m})$.  Observe that $g_{p,q}$ is not smooth if $p$ or $q$ is negative. In particular, there is a singularity at $l(\alpha)=0$ when $p<0$, while $g_{p,q}(l(\alpha),0)$ is singular when $q<0$.

The strategy is as follows for constructing equivariant versions of mixed $\psi$ and $\kappa_1$ classes. Consider the commutative (up to homotopy) square
\begin{equation*}
  \begin{CD}
   \CompactModuli^{\Gamma}_{g,n}   @>\Phi_L>>	\CompactModuli^{\Gamma}_{g,n}(L) \\
	@VVV		@VVV \\
   \CompactModuli_{g,n}		@>\phi_L>>	\CompactModuli_{g,n}(L),
  \end{CD}
\end{equation*}
where $\phi_L$ is a family of diffeomorphisms and $\Phi_L$ is a family of equivariant diffeomorphisms. Existence of these families of maps follows from the triviality of the bundles $\widehat{\Moduli}_{g,n}\rightarrow \Reals^n_{\geq 0}$ and $\widehat{\Moduli}^{\Gamma}_{g,n}\rightarrow\Reals^n_{\geq 0}$. By work of Mirzakhani~\cite{art:MirzakhaniIT}, the pullback forms $\phi^*_L \omega_{WP}$ are cohomologous to $\omega_{WP} + \frac{1}{2}\sum L_i^2 \psi_i$. Hence 
\begin{equation}
  2^{k_1 + \cdots k_n} \frac{k_0!\cdots k_n!}{(2k_1)!\cdots(2k_n)!} \partial_{L_1}^{2k_1} \cdots \partial_{L_n}^{2k_n} \frac{\phi^*_L \omega_{WP}^d}{d!} \Bigr|_{L=0} = \omega_{WP}^{k_0}\psi_1^{k_1}\cdots\psi_n^{k_n}.
  \label{eqn:TautDeriv}
\end{equation}

Using McShane's identity on $\CompactModuli_{g,n}(L)$ we lift $\frac{1}{d!}\omega_{WP}^d$ to differential forms \\ $\mathcal{R}(L_1, L_j, l(\gamma))\frac{\omega_{WP}^d}{d!}$ on $\CompactModuli^{\gamma}_{g,n}(L)$ and $\mathcal{D}(L_1, l(\alpha), l(\beta))\frac{\omega_{WP}^d}{d!}$ on $\CompactModuli^{\alpha,\beta}_{g,n}(L)$. These forms have equivariant extensions, using \eqref{eqn:EquivSingle} and \eqref{eqn:EquivDouble} respectively. Finally, pulling back by the equivariant map $\Phi_{L}$ and taking derivatives as in \eqref{eqn:TautDeriv} results in equivariantly closed forms on $\CompactModuli^{\gamma}_{g,n}$ and $\CompactModuli^{\alpha,\beta}_{g,n}$. It is these forms to which we apply the localization theorem.

The following calculations allow us to deduce Mirzakhani's recursion relation~\cite{art:MirzakhaniWP} from the localization techniques of this paper.
Let $\mathcal{P}_x$ be the integral operator $\mathcal{P}f(x) = \int_{x}^{\infty}tf(t)dt$. Hence $\mathcal{P}_x^n f(x) = f_k(x)$, while $\mathcal{P}_x^n\mathcal{P}_y^m g(x,y)=g_{n,m}(x,y)$.
\begin{proposition}
  $\mathcal{P}^n e^{-\alpha x} = \sum_{j=0}^{n} A_j^{(n)} \frac{x^j}{\alpha^{2n-j}} e^{-\alpha x}$, where
  \begin{equation*}
    A_k^{(n)} = \frac{(2n-k)!}{2^{n-k}k!(n-k)!}
  \end{equation*}
  \label{prop:PRelation}
\end{proposition}
\begin{proof}
  The proof is by induction on $k+n$ using the identity
  \begin{equation*}
    A_{k}^{(n)} = \frac{A_{k-1}^{(n)} - A_{k-2}^{(n-1)}}{k}.
  \end{equation*}
  This identity stems from the relation
  \begin{equation*}
    -\frac{1}{x}\frac{d}{dx} \mathcal{P}^n f(x) = \mathcal{P}^{n-1}f(x).
  \end{equation*}
\end{proof}

\begin{corollary}
  If $h(x) = \frac{2}{1 + e^{x/2}}$ then 
  \begin{multline*}
    \mathcal{P}_x^n\mathcal{P}_y^m h^{(2k)} (x+ y)\Bigr|_{x,y=0}  \\
      = \begin{cases} 
	(2n-1)!!(2m-1)!!(2^{2(n+m-k)+1}-4 )\zeta(2(n+m-k)) & k < n+m \\
	(2n-1)!!(2m-1)!!	& k = n+m \\
  	0	& k > n+m.
                                            \end{cases}
  \end{multline*}
\end{corollary}
 \begin{proof}
The case $k < n+m$ follows immediately from Proposition~\ref{prop:PRelation}. The other cases are derived from the relations $[\mathcal{P}, \partial^2] = 2$ and
\begin{equation*}
   \partial^{2k}\mathcal{P}^n h(0) = \begin{cases}
 	(-1)^k(2k-1)!! \mathcal{P}^{n-k} h(0)  & k \leq n \\
 	0 & k > n.
                                     \end{cases}
 \end{equation*}
Finally, note that $k>n+m$ gives 0 because $h(x)$ is an odd function plus a constant.
 \end{proof}

Combining the above calculations and using the Localization Theorem~\ref{thm:OpenLocalization} we can calculate intersection numbers
\begin{equation*}
	\intersect{\kappa_1^{k_0}\tau_{k_1}\cdots\tau_{k_n}}_g \stackrel{\text{def}}{=} 
	\frac{1}{(2\pi^2)^{k_0}}\int_{\CompactModuli_{g,n}}\omega_{WP}^{k_0}\psi_1^{k_1}\cdots \psi_{n}^{k_n}.
\end{equation*}
by the rule given below.
\begin{proposition}
\begin{multline*}
  (2k_1+1)!!\intersect{\kappa_1^{k_0}\tau_{k_1}\cdots\tau_{k_n}}_g \\
  =  \sum_{j=2}^{n} \sum_{l=0}^{k_0}\frac{k_0!}{(k_0-l)!} \frac{(2(l+k_1+k_j)-1)!!}{(2k_j-1)!!} \beta_l \intersect{\kappa_1^{k_0-l}\tau_{k_1+k_j+l-1}\prod_{i \neq 1,j}\tau_{k_i}}_g \\
 + \frac{1}{2} \sum_{l=0}^{k} \sum_{d_1 + d_2 = l+k_1 - 2}  \frac{k_0!}{(k_0-l)!}
  (2d_1+1)!! (2d_2+1)!! \beta_l \intersect{\kappa_1^{k_0-l}\tau_{d_1}\tau_{d_2}\prod_{i\neq 1}\tau_{k_i}}_{g-1}  \\
  + \frac{1}{2}\sum_{\substack{g_1+g_2 = g \\ \mathcal{I} \coprod \mathcal{J} = \{2, \ldots, n \}}}
  \sum_{l=0}^{k_0}\sum_{d_1 + d_2 = l+k_1-2}  \frac{k_0!}{m_0! n_0!}
 (2d_1+1)!!(2d_2+1)!! \beta_l \\
  \times \intersect{\kappa_1^{m_0}\tau_{d_1}\tau_{k(\mathcal{I})}}_{g_1}
 \intersect{\kappa_1^{n_0}\tau_{d_2}\tau_{k(\mathcal{J})}}_{g_2},
\end{multline*}
where
\begin{equation*}
\beta_l = (2^{2l+1}-4)\frac{\zeta(2l)}{(2\pi^2)^l} = (-1)^{l-1}2^l(2^{2l}-2) \frac{B_{2l}}{(2l)!},
\end{equation*}
and $m_0$, $n_0$ in the third summand are the unique integers with $m_0 + n_0 = k_0-l$ and each intersection number having the correct dimension.
\end{proposition}
As proven in \cite{math.QA/0601194}, this formula is an equivalent differential version of Mirzakhani's recursion relation and is a Virasoro constraint condition on the generating function
\begin{equation*}
  G(s, t_0, t_1, \ldots) = \sum_g \intersect{e^{s\kappa_1 + t_0 \tau_0 + \cdots}}_g
\end{equation*}
of mixed $\psi$ and $\kappa_1$ intersection numbers.

\bibliographystyle{hamsplain}
\bibliography{References}

\end{document}